\documentclass[11pt]{article}
\usepackage{amsmath,amssymb,amsthm}
\usepackage{tikz}
\usetikzlibrary{calc}
\usepackage{float}
\usepackage{xcolor}
\usepackage{authblk}
\usepackage{autobreak}
\usepackage{enumerate}
\usepackage{hyperref}
\hypersetup{
	colorlinks,
	linkcolor={red!60!black},
	citecolor={green!60!black},
	urlcolor={blue!60!black}
}

\newtheorem{theorem}{Theorem}
\newtheorem{lemma}{Lemma}[section]

\theoremstyle{definition}

\newtheorem{claim}{Claim}[section]
\newtheorem{conjecture}{Conjecture}

\begin{document}

\title{Disproofs of four Gallai-Ramsey-type conjectures}
\author[1]{Yanbo Zhang}
\affil[1]{School of Mathematical Sciences\\ Hebei Normal University\\ Shijiazhuang 050024, China}
\author[2]{Yaojun Chen}
\affil[2]{School of Mathematics\\ Nanjing University\\ Nanjing 210093, China}

\date{}
\maketitle
\let\thefootnote\relax\footnotetext{\emph{Email addresses:} {\tt ybzhang@hebtu.edu.cn} (Yanbo Zhang), {\tt yaojunc@nju.edu.cn} (Yaojun Chen)}
\begin{quote}
{\bf Abstract:}
As a significant variation of Ramsey numbers, the Gallai-Ramsey number $GR_k(H)$ refers to the smallest positive integer $r$ such that, by coloring the edges of $K_r$ with at most $k$ colors, there exists either a monochromatic subgraph isomorphic to $H$ or a rainbow triangle. Mao, Wang, Magnant, and Schiermeyer [Discrete Math., 2023], Song, Wei, Zhang, and Zhao [Discrete Math., 2020], and Zhao and Wei [Discrete Appl. Math., 2021] each proposed one conjecture on the Gallai-Ramsey numbers for fans, wheels, and kipases, respectively. We establish new lower bounds that disprove all three conjectures. Su and Liu [Graphs Combin., 2022] studied the Gallai-Ramsey-full property of graphs and conjectured that a graph is Ramsey-full if and only if it is Gallai-Ramsey-full. We present two classes of graphs that are Ramsey-full, but neither is Gallai-Ramsey-full.

{\bf Keywords:} Ramsey number, Gallai-Ramsey number, Ramsey-full, Gallai-Ramsey-full

{\bf AMS Subject Classification:} 05C55, 05D10
\end{quote}

\section{Introduction}

We use $[q,k]$ to represent the set $\{q, q+1, \ldots, k\}$, and abbreviate $[1,k]$ as $[k]$. Given $k$ graphs $G_1, \ldots, G_k$, the \emph{Ramsey number} $R(G_1, G_2, \ldots, G_k)$ is defined as the smallest positive integer $n$ such that for any edge coloring $\tau: E(K_n) \to [k]$, there exists a monochromatic subgraph $G_i$ of color $i$ for some $i \in [k]$.

When $k=2$ and both $G_1$ and $G_2$ are complete graphs, we refer to $R(K_p, K_q)$ as the classical Ramsey number, often abbreviated as $R(p, q)$. Currently, only nine exact values of non-trivial classical Ramsey numbers are known: $R(3, q)$ for $q \in [3, 9]$, $R(4, 4)$, and $R(4, 5)$. The most famous unresolved problem is $R(5, 5)$, which gained further popularity due to a famous joke by Erd\H{o}s~\cite{Erdos1985}. The best known bounds are 

\[
43 \le R(5, 5) \le 46\,.
\]

The lower bound was established by Exoo~\cite{Exoo1989}, while the upper bound was recently obtained by Angeltveit and McKay~\cite{Angeltveit2024}, utilizing 80 years of CPU time. McKay and Radziszowski~\cite{McKay1997} conjectured that the lower bound is the exact value and provided some experimental evidence.

\begin{conjecture}[McKay and Radziszowski, 1997]
  $R(5, 5) = 43$.
\end{conjecture}

Ramsey theory has developed many variations. One such variation restricts the way of coloring, making it relatively easier to study the multicolor case. An edge coloring $\tau: E(G) \to [k]$ is called a Gallai coloring of the graph $G$ if it does not contain a rainbow triangle, meaning no triangle has all three edges colored with different colors. This type of coloring was first studied by Gallai in 1967 \cite{Gallai1967}, who established a significant theorem in this field.

\begin{theorem}\label{Gallai}
    Let $(G, \tau)$ be a Gallai $k$-colored complete graph with $|V(G)| \ge 2$. Then, the vertex set $V(G)$ can be partitioned into nonempty sets $V_1, \dots, V_p$ with $p \ge 2$ such that at most two colors are used on the edges in $E(G) \setminus (E(V_1) \cup \cdots \cup E(V_p))$, and only one color is used on the edges between any fixed pair $(V_i, V_j)$ under $\tau$, where $E(V_i)$ denotes the set of edges with both ends in $V_i$ for all $i \in [p]$. 
\end{theorem}

Accordingly, we define the Ramsey number under this restricted coloring scheme. The \emph{Gallai-Ramsey number} $GR(G_1, \ldots, G_k)$ is the smallest positive integer $N$ such that for any Gallai coloring $\tau: E(K_N) \to [k]$, there exists a subgraph of color $i$ isomorphic to $G_i$ for some $i \in [k]$. When $G_1 = \cdots = G_k = G$, we abbreviate it as $GR_k(G)$.

Although we know almost nothing about the multicolor Ramsey numbers for complete graphs, the Gallai-Ramsey numbers for complete graphs may behave nicely. Fox, Grinshpun, and Pach~\cite{Fox2015} proposed the following conjecture.

\begin{conjecture}[Fox, Grinshpun, and Pach, 2015]
  For positive integers $k$ and $p$,
  \begin{equation*}
    GR_k(K_p)= 
    \begin{cases} 
    (R(K_p)-1)^{k/2}+1, & \text{if } k \text{ is even};\\
    (p-1)(R(K_p)-1)^{(k-1)/2}+1, & \text{if } k \text{ is odd}.
    \end{cases}
    \end{equation*}
\end{conjecture}

For $p=3$, this result was first established by Chung and Graham~\cite{Chung1983} and was later reproved by Gy\'arf\'as, S\'ark\"ozy, Seb\H{o}, and Selkow~\cite{Gyarfas2010} using a simpler approach. For $p=4$, the conjecture was confirmed by Liu, Magnant, Saito, Schiermeyer, and Shi~\cite{Liu2020}. In the case of $p=5$, significant progress was made by Magnant and Schiermeyer~\cite{Magnant2022}, although the precise value remains unresolved. Furthermore, they highlighted a notable tension between the conjecture proposed by Fox, Grinshpun, and Pach~\cite{Fox2015} and the one by McKay and Radziszowski~\cite{McKay1997}, suggesting that only one of these conjectures can hold true.

Now let us shift our focus to the Gallai-Ramsey numbers of sparse graphs. Currently, only a few exact values of multicolor Ramsey numbers for sparse graphs are known. However, when restricted to the framework of Gallai coloring, many corresponding Gallai-Ramsey numbers can achieve exact values or meaningful bounds. This provides a new perspective for understanding the multicolor Ramsey numbers of sparse graphs.

We study three classes of graphs that all contain a central vertex. The graph $K_1 + G$ is constructed by adding an extra vertex to a graph $G$, where this extra vertex is connected to every vertex of $G$ by an edge. This extra vertex is referred to as the central vertex. When $G$ is a matching, $K_1 + nK_2$ is called a fan, denoted as $F_n$. When $G$ is a path with $n$ vertices, $K_1 + P_n$ is called a kipas, denoted as $\widehat{K}_n$. When $G$ is a cycle with $n$ vertices, $K_1 + C_n$ is called a wheel, denoted as $W_n$.

For the Gallai-Ramsey number of fans, Mao, Wang, Magnant, and Schiermeyer~\cite{Mao2023} calculated the exact value of $GR_k(F_2)$ and provided a lower bound for $GR_k(F_3)$. They conjectured that this lower bound is indeed the exact value.
\begin{conjecture}
  For $k \ge 2$,
  \begin{equation*}
    GR_k(F_3) = 
    \begin{cases} 
    14 \cdot 5^{\frac{k-2}{2}}-1, & \text{if } k \text{ is even};\\
    33 \cdot 5^{\frac{k-3}{2}}, & \text{if } k \text{ is odd}.
    \end{cases}
    \end{equation*}
\end{conjecture}

When $k$ is even, we provide a new lower bound for $GR_k(F_3)$, thus disproving this conjecture.
\begin{theorem}
  For an even integer $k$,
  \begin{equation*}
    GR_k(F_3) \ge 
    \begin{cases} 
    14, & \text{if } k = 2;\\
    14 \cdot 5^{\frac{k-2}{2}} + 1, & \text{if } k \ge 4.
    \end{cases}
    \end{equation*}
\end{theorem}

For the Gallai-Ramsey number of kipases, Zou, Mao, Magnant, Wang, and Ye~\cite{Zou2019} provided the exact value of $GR_k(\widehat{K}_3)$. Zhao and Wei~\cite{Zhao2021} determined the exact value of $GR_k(\widehat{K}_4)$. For general kipases, they further proposed the following conjecture.
\begin{conjecture}\label{conjkipas}
  For all $k \ge 2$ and $m \ge 2$,
  \begin{equation*}
    GR_k(\widehat{K}_m) =
    \begin{cases}
    (R_2(\widehat{K}_m) - 1) \cdot 5^{\frac{k-2}{2}} + 1, & \text{if } k \text{ is even and } m \text{ is odd},\\
    R_2(\widehat{K}_m) + \frac{m}{2} \cdot (5^{\frac{k}{2}} - 5), & \text{if } k \text{ is even and } m \text{ is even},\\
    \max\{2(R_2(\widehat{K}_m) - 1), 5m\} \cdot 5^{\frac{k-3}{2}} + 1, & \text{if } k \text{ is odd}.
    \end{cases}
  \end{equation*}	
\end{conjecture}

For the Gallai-Ramsey number of wheels, Song, Wei, Zhang, and Zhao~\cite{Song2020} and Mao, Wang, Magnant, and Schiermeyer~\cite{Mao2023} independently provided the exact value of $GR_k(W_4)$. For general odd-order wheels, the former paper proposed the following conjecture.
\begin{conjecture}\label{conjwheel}
  For all $k \ge 2$ and even $m \ge 4$,
    \begin{equation*}
      GR_k(W_m) = 
      \begin{cases} 
      (R_2(W_m)-1)\cdot 5^{\frac{k-2}{2}}+1, & \text{if } k \text{ is even};\\
      2(R_2(W_m)-1)\cdot 5^{\frac{k-3}{2}}+1, & \text{if } k \text{ is odd}.
      \end{cases}
    \end{equation*}
\end{conjecture}

We negate both Conjectures~\ref{conjkipas} and~\ref{conjwheel} with a unified result.

\begin{theorem}\label{thmkipas}
  For all $k \ge 2$ and even $m \ge 6$, regardless of whether $G$ is $\widehat{K}_m$ or $W_m$, the following inequalities hold.
    \begin{equation*}
      GR_k(G) \ge  
      \begin{cases}
      \max\{2R_2(G)-1, 5m+1\}, & \text{if } k=3;\\
      R_2(G)+5(m-2)(5^{\frac{k-2}{2}}-1)-(m-4)(k-2), & \text{if } k \text{ is even and } m \equiv 0 \pmod{4};\\
      R_2(G)+5(m-1)(5^{\frac{k-2}{2}}-1)-(m-2)(k-2), & \text{if } k \text{ is even and } m \equiv 2 \pmod{4};\\
      2(R_2(G)+5(m-2)(5^{\frac{k-3}{2}}-1)-(m-4)(k-3))-1, & \text{if odd } k \ge 5 \text{ and } m \equiv 0 \pmod{4};\\
      2(R_2(G)+5(m-1)(5^{\frac{k-3}{2}}-1)-(m-2)(k-3))-1, & \text{if odd } k \ge 5 \text{ and } m \equiv 2 \pmod{4}.
      \end{cases}
    \end{equation*}
\end{theorem}

Readers may note that when $k=2$ or $k=3$, the current lower bound coincides with the original lower bound. However, when $k \ge 4$, as the number of colors increases, the flexibility of the coloring schemes also expands. The new lower bound significantly surpasses the original lower bound.

Next, we calculate the difference between the new lower bound and the original lower bound for $GR_k(\widehat{K}_6)$. When $k$ is an even number of at least $4$, this difference is given by
\[
2\cdot 5^{\frac{k}{2}}-4k-2.
\]
When $k$ is an odd number of at least $5$, this difference is
\[
\min\{(52-2R_2(\widehat{K}_6))5^{\frac{k-3}{2}}+2R_2(\widehat{K}_6), 2R_2(\widehat{K}_6)+4\cdot 5^{\frac{k-1}{2}}\}-8k-28.
\]
Recall that $R_2(\widehat{K}_6) \le R_2(W_6) = 19$~\cite{Lidicky2021}. It is evident that both of the above differences are always positive. Furthermore, as $k$ increases, these differences grow rapidly.

We also calculate the difference between the new lower bound and the original lower bound for $GR_k(W_6)$. Once again, we need that fact that $R_2(W_6) = 19$. When $k$ is an even number of at least $4$, this difference is
\[
7\cdot 5^{\frac{k}{2}}-4k+1.
\]
When $k$ is an odd number of at least $5$, this difference is
\[
14\cdot 5^{\frac{k-3}{2}}-8k+10.
\]
Again, it is evident that both of these differences are always positive, and they increase rapidly with larger $k$.

When $m$ is an even number of at least $8$, although we have not yet determined the exact values of $R_2(\widehat{K}_m)$ and $R_2(W_m)$, we proved in another paper~\cite{Zhang2024} that
\[
R_2(\widehat{K}_m) \le 4m-2\ \text{and}\ R_2(W_m) \le 4m+\min\left\{m/2, 664\right\}.
\]
By comparing the coefficients of the exponential terms in the new and original bounds, we conclude that the new bounds are evidently superior to the old ones for large $m$.

Notably, in the case of kipas $\widehat{K}_m$, $m$ can be either even or odd. When $m$ is odd, we can make the following improvements.

\begin{theorem}\label{thmoddm}
  For all $k \ge 2$ and odd $m \ge 7$, the following inequalities hold:
  \begin{equation*}
    GR_k(\widehat{K}_m) \ge  
    \begin{cases}
    \max\{2R_2(\widehat{K}_m)-1,5m+1\}, & \text{if } k=3;\\
    R_2(\widehat{K}_m)+5(m-3)(5^{\frac{k-2}{2}}-1)-(m-5)(k-2), & \text{if } k \text{ is even and } m \equiv 1 \pmod{4};\\
    R_2(\widehat{K}_m)+5(m-2)(5^{\frac{k-2}{2}}-1)-(m-3)(k-2), & \text{if } k \text{ is even and } m \equiv 3 \pmod{4};\\
    2(R_2(\widehat{K}_m)+5(m-3)(5^{\frac{k-3}{2}}-1)-(m-5)(k-3))-1, & \text{if odd } k \ge 5 \text{ and } m \equiv 1 \pmod{4};\\
    2(R_2(\widehat{K}_m)+5(m-2)(5^{\frac{k-3}{2}}-1)-(m-3)(k-3))-1, & \text{if odd } k \ge 5 \text{ and } m \equiv 3 \pmod{4}.
    \end{cases}
  \end{equation*}
\end{theorem}

Now we turn to the final conjecture. First, we need to clarify some concepts. Together with Broersma, we introduced the concept of Ramsey-full in \cite{Zhang2016}. Let $R(G_1,G_2)=N$. If there exists a red-blue edge-coloring of the graph $K_N-e$ such that it contains neither a red copy of $G_1$ nor a blue copy of $G_2$, then the pair $(G_1,G_2)$ is called \emph{Ramsey-full}. When $G_1=G_2$, we say that the graph $G_1$ is Ramsey-full.

Su and Liu~\cite{Su2022} proposed a corresponding concept within Gallai-Ramsey theory. Let $GR(G_1, \ldots, G_k)=N$. If there exists a Gallai $k$-edge-coloring of $K_N-e$ using colors from $[k]$ such that for any $i \in [k]$, there is no $G_i$ in color $i$, then the set of graphs $(G_1, \ldots, G_k)$ is said to be \emph{Gallai-Ramsey-full}. When $G_1=\cdots=G_k$, the graph $G_1$ is referred to as Gallai-Ramsey-full.

Su and Liu~\cite{Su2022} noted that complete graphs and cycles of length four have the property that they are both Ramsey-full and Gallai-Ramsey-full. Based on this observation, they proposed the following conjecture.

\begin{conjecture}[Su and Liu, 2022]\label{Suconjecture}
  Let $H$ be a graph with no isolated vertex. Then $H$ is Ramsey-full if and only if $H$ is Gallai-Ramsey-full.
\end{conjecture}

We use $K_{1,n}$ to denote the star on $n+1$ vertices, and $K^+_{1,n}$ to represent the graph obtained by subdividing one edge of $K_{1,n}$. Hence, $K^+_{1,n}$ is a tree with $n+2$ vertices and a maximum degree of $n$. The following theorem disproves the aforementioned conjecture.

\begin{theorem}\label{Ramseyfull}
  For even $n \ge 12$, both $K_{1,n}$ and $K^+_{1,n}$ are Ramsey-full, but neither is Gallai-Ramsey-full.
\end{theorem}

The remaining content will be organized as follows. In Section~\ref{secfan}, we will address the conjecture concerning the Gallai-Ramsey number of $F_3$. The conjectures related to the Gallai-Ramsey numbers of kipases and wheels will be discussed in Section~\ref{seckipas}. Finally, in the last section of this paper, we will refute the conjecture concerning the relationship between Ramsey-full graphs and Gallai-Ramsey-full graphs.

\section{Gallai-Ramsey numbers of fans}\label{secfan}

\begin{figure}[H]
  \centering
  \scalebox{0.7}{ 
  \begin{tikzpicture}

  \def\radius{2cm}

  \def\pentagonRadius{3cm}
  
  \def\lineWidth{15pt}

  \coordinate (A) at ({\pentagonRadius*cos(90)}, {\pentagonRadius*sin(90)});
  \coordinate (B) at ({\pentagonRadius*cos(162)}, {\pentagonRadius*sin(162)});
  \coordinate (C) at ({\pentagonRadius*cos(234)}, {\pentagonRadius*sin(234)});
  \coordinate (D) at ({\pentagonRadius*cos(306)}, {\pentagonRadius*sin(306)});
  \coordinate (E) at ({\pentagonRadius*cos(18)}, {\pentagonRadius*sin(18)});

  \draw[red!50, line width=\lineWidth] (A) -- (B);
  \draw[red!50, line width=\lineWidth] (B) -- (C);
  \draw[red!50, line width=\lineWidth] (C) -- (D);
  \draw[red!50, line width=\lineWidth] (D) -- (E);
  \draw[red!50, line width=\lineWidth] (E) -- (A);

  \draw[blue!50, line width=\lineWidth] (A) -- (C);
  \draw[blue!50, line width=\lineWidth] (C) -- (E);
  \draw[blue!50, line width=\lineWidth] (E) -- (B);
  \draw[blue!50, line width=\lineWidth] (B) -- (D);
  \draw[blue!50, line width=\lineWidth] (D) -- (A);

  \foreach \i in {A,B,C,D,E}{
      \node[draw, fill=white, circle, minimum size=\radius] at (\i) {};
  }

  \def\vertexSize{8pt}

  \begin{scope}
    \coordinate (V1) at ($(A)+(-0.5,0)$);
    \coordinate (V2) at ($(A)+(0.5,0)$);
    \draw[line width=2pt, black] (V1) -- (V2);
    \node at (V1) [draw, circle, fill=white, minimum size=\vertexSize, inner sep=0pt] {};
    \node at (V2) [draw, circle, fill=white, minimum size=\vertexSize, inner sep=0pt] {};
  \end{scope}

  \begin{scope}
    \coordinate (V3) at ($(B)+(-0.5,-0.3)$);
    \coordinate (V4) at ($(B)+(0.5,-0.3)$);
    \coordinate (V5) at ($(B)+(0,{sqrt(3)/2-0.3})$);
    \draw[red, line width=2pt] (V3) -- (V4);
    \draw[red, line width=2pt] (V4) -- (V5);
    \draw[red, line width=2pt] (V5) -- (V3);
    \node at (V3) [draw, circle, fill=white, minimum size=\vertexSize, inner sep=0pt] {};
    \node at (V4) [draw, circle, fill=white, minimum size=\vertexSize, inner sep=0pt] {};
    \node at (V5) [draw, circle, fill=white, minimum size=\vertexSize, inner sep=0pt] {};    
  \end{scope}

  \begin{scope}
    \coordinate (V6) at ($(C)+(-0.5,-0.3)$);
    \coordinate (V7) at ($(C)+(0.5,-0.3)$);
    \coordinate (V8) at ($(C)+(0,{sqrt(3)/2-0.3})$);
    \draw[blue, line width=2pt] (V6) -- (V7);
    \draw[blue, line width=2pt] (V7) -- (V8);
    \draw[blue, line width=2pt] (V8) -- (V6);
    \node at (V6) [draw, circle, fill=white, minimum size=\vertexSize, inner sep=0pt] {};
    \node at (V7) [draw, circle, fill=white, minimum size=\vertexSize, inner sep=0pt] {};
    \node at (V8) [draw, circle, fill=white, minimum size=\vertexSize, inner sep=0pt] {};    
  \end{scope}

  \begin{scope}
    \coordinate (V9) at ($(D)+(-0.5,-0.3)$);
    \coordinate (V10) at ($(D)+(0.5,-0.3)$);
    \coordinate (V11) at ($(D)+(0,{sqrt(3)/2-0.3})$);
    \draw[blue, line width=2pt] (V9) -- (V10);
    \draw[blue, line width=2pt] (V10) -- (V11);
    \draw[blue, line width=2pt] (V11) -- (V9);
    \node at (V9) [draw, circle, fill=white, minimum size=\vertexSize, inner sep=0pt] {};
    \node at (V10) [draw, circle, fill=white, minimum size=\vertexSize, inner sep=0pt] {};
    \node at (V11) [draw, circle, fill=white, minimum size=\vertexSize, inner sep=0pt] {};
  \end{scope}

  \begin{scope}
    \coordinate (V12) at ($(E)+(-0.5,-0.3)$);
    \coordinate (V13) at ($(E)+(0.5,-0.3)$);
    \coordinate (V14) at ($(E)+(0,{sqrt(3)/2-0.3})$);
    \draw[red, line width=2pt] (V12) -- (V13);
    \draw[red, line width=2pt] (V13) -- (V14);
    \draw[red, line width=2pt] (V14) -- (V12);
    \node at (V12) [draw, circle, fill=white, minimum size=\vertexSize, inner sep=0pt] {};
    \node at (V13) [draw, circle, fill=white, minimum size=\vertexSize, inner sep=0pt] {};
    \node at (V14) [draw, circle, fill=white, minimum size=\vertexSize, inner sep=0pt] {};
  \end{scope}
\end{tikzpicture}}
\caption{The graph $G_3^i$, where black represents color $i$.}
\label{fig:fan}
\end{figure}
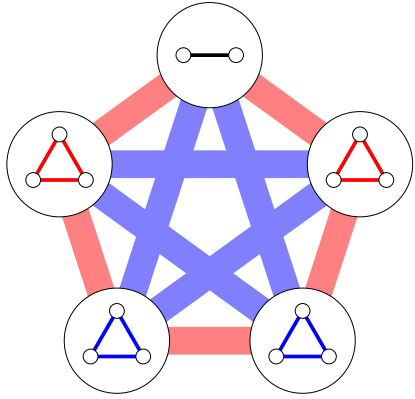

When $k=2$, we need the following result. In fact, Figure~\ref{fig:fan} is derived from a slight modification of the Ramsey graph of $R(F_3,F_3)$ in \cite{Zhao2022}. If the black edge in Figure~\ref{fig:fan} is contracted into a single vertex, the resulting graph is the Ramsey graph of $R(F_3,F_3)$.

\begin{lemma}[Zhao and Wei~\cite{Zhao2022}]
  $R(F_3,F_3) = 14$.
\end{lemma}

For even $k \ge 4$, we can raise the lower bound of $GR_k(F_3)$ by using the Ramsey graph of $(F_3, F_3)$. Let $(K_5, c)$ denote the complete graph on five vertices with a red-blue edge coloring, where both the red and blue edges induce a 5-cycle, and the red 5-cycle is $v_1v_2v_3v_4v_5v_1$. Now, we replace the red and blue colors in $(K_5, c)$ with colors $1$ and $2$, respectively. Expand $v_2$ and $v_5$ into a triangle of color $1$, expand $v_3$ and $v_4$ into a triangle of color $2$, and expand $v_1$ into a $K_2$ of color $i$, where $3 \le i \le 4$. This graph is denoted by $G_3^i$, as shown in Figure~\ref{fig:fan}. It can be easily verified that $|G_3^i|=14$, and $G_3^i$ does not contain a monochromatic $F_3$.

Next, we replace the red and blue colors in $(K_5, c)$ with colors $3$ and $4$, respectively. Then expand $v_1$, $v_3$, and $v_4$ into copies of $G_3^3$, and expand $v_2$ and $v_5$ into copies of $G_3^4$. This graph is denoted by $G_4$. Clearly, $|G_4|=70$. A triangle of color $3$ must include an edge of color $3$ from $G_3^3$. If $G_4$ contains an $F_3$ of color $3$, this $F_3$ must contain three color-3 edges from the three copies of $G_3^3$. Based on the relative positions of the three $G_3^3$ copies, this is impossible. Similarly, we can verify that $G_4$ does not contain an $F_3$ of color $4$. Therefore,
\[
GR_4(F_3) \ge 71\,.
\]

For $k \ge 6$, the construction of $G_k$ is as follows. First, replace the red and blue colors in $(K_5, c)$ with colors $k-1$ and $k$, respectively. Then expand each vertex of $K_5$ into a copy of $G_{k-2}$. It can be easily verified that this graph contains no monochromatic $F_3$. Therefore,
\[
GR_k(F_3) \ge 14 \cdot 5^{\frac{k-2}{2}} + 1 \text{ for even } k \ge 4\,.
\]

\section{Gallai-Ramsey numbers of kipases and wheels}\label{seckipas}

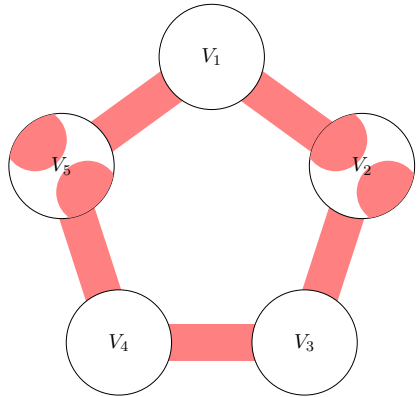
\begin{figure}[H]
  \centering
  \scalebox{0.7}{
  \begin{tikzpicture}

    \def\radius{2cm}

    \def\pentagonRadius{3cm}

    \def\lineWidth{20pt}
  
    \coordinate (A) at ({\pentagonRadius*cos(90)}, {\pentagonRadius*sin(90)});
    \coordinate (B) at ({\pentagonRadius*cos(162)}, {\pentagonRadius*sin(162)});
    \coordinate (C) at ({\pentagonRadius*cos(234)}, {\pentagonRadius*sin(234)});
    \coordinate (D) at ({\pentagonRadius*cos(306)}, {\pentagonRadius*sin(306)});
    \coordinate (E) at ({\pentagonRadius*cos(18)}, {\pentagonRadius*sin(18)});
  
    \draw[red!50, line width=\lineWidth] (A) -- (B);
    \draw[red!50, line width=\lineWidth] (B) -- (C);
    \draw[red!50, line width=\lineWidth] (C) -- (D);
    \draw[red!50, line width=\lineWidth] (D) -- (E);
    \draw[red!50, line width=\lineWidth] (E) -- (A);

    \foreach \i in {A,B,C,D,E}{
        \node[draw, fill=white, circle, minimum size=\radius] at (\i) {};
    }

    \begin{scope}
        \clip (B) circle (\radius/2);
        \fill[red!50] ($(B)+(-0.5,0.5)$) circle (0.6cm);
        \fill[red!50] ($(B)+(0.5,-0.5)$) circle (0.6cm);
    \end{scope}
  
    \begin{scope}
        \clip (E) circle (\radius/2);
        \fill[red!50] ($(E)+(-0.5,0.5)$) circle (0.6cm);
        \fill[red!50] ($(E)+(0.5,-0.5)$) circle (0.6cm);
    \end{scope}
  
    \node at (A) {$V_1$};
    \node at (B) {$V_5$};
    \node at (C) {$V_4$};
    \node at (D) {$V_3$};
    \node at (E) {$V_2$};
  \end{tikzpicture}}
  \caption{The graph of Lemma~\ref{basiclemma}}
  \label{fig:lemma2}
\end{figure}

The following lemma essentially illustrates why we can add vertices to establish a better lower bound.
\begin{lemma}\label{basiclemma}
  We perform a blow-up procedure on a red 5-cycle $v_1v_2v_3v_4v_5v_1$. For each $i \in [5]$, let the clique formed by the expansion of vertex $v_i$ be denoted as $V_i$. In the subgraphs $G[V_1]$, $G[V_3]$, and $G[V_4]$, there are no red edges. Meanwhile, in the subgraphs $G[V_2]$ and $G[V_5]$, the maximum degree of the subgraph induced by the red edges is at most $m/2-1$, and each red connected component contains at most $m-1$ vertices (as illustrated in Fig.~\ref{fig:lemma2}). Then, there is no red $\widehat{K}_m$ in the entire graph.
\end{lemma}

\begin{proof}
  We proceed by contradiction. Assume there exists a red $\widehat{K}_m$ in the graph. Let $v$ be the center vertex of $\widehat{K}_m$. If $v$ is in $V_1$, the vertices connected to $v$ by red edges must lie in $V_2 \cup V_5$. Hence, the non-center vertices of $\widehat{K}_m$ forming $P_m$ must be contained entirely in either $V_2$ or $V_5$. According to the given conditions, each red connected component in $V_2$ and $V_5$ contains no more than $m-1$ vertices. This contradiction shows that $v$, the center of $\widehat{K}_m$, cannot be in $V_1$. A similar argument proves that $v$ is not in $V_3 \cup V_4$. Thus, the center vertex $v$ must be in either $V_2$ or $V_5$. Without loss of generality, assume $v \in V_2$.
  
  Among the non-center vertices of $\widehat{K}_m$ forming $P_m$, each independent set contains at most $m/2$ vertices. Since there are no red edges in $G[V_1]$ and $G[V_4]$, at least $m/2$ vertices of $P_m$ must be in $V_2$. In other words, $v$ in $V_2$ must be connected by at least $m/2$ red edges, which contradicts the given conditions. 
\end{proof}

Since $GR_k(W_m) \ge GR_k(\widehat{K}_m)$, to prove Theorem~\ref{thmkipas}, it suffices to establish the theorem for $\widehat{K}_m$. We will begin by proving the case when $k$ is even.

\begin{lemma}
  For even $k \ge 2$ and even $m \ge 6$, we have
    \begin{equation*}
      GR_k(\widehat{K}_m) \ge  
      \begin{cases}
      R_2(\widehat{K}_m)+5(m-2)(5^{\frac{k-2}{2}}-1)-(m-4)(k-2), & \text{if } k \text{ is even and } m \equiv 0\ (\bmod{\ 4});\\
      R_2(\widehat{K}_m)+5(m-1)(5^{\frac{k-2}{2}}-1)-(m-2)(k-2), & \text{if } k \text{ is even and } m \equiv 2\ (\bmod{\ 4}).
      \end{cases}
    \end{equation*}
\end{lemma}
\begin{proof}
We will construct six types of graphs, namely $G(i,j)$, $H_{2k}$, $F(p,q)$, $H_{2k}[F(p,q)]$, $G_{2k-2}(\ell)$, and $G_{2k}$. Among these, $G_{2k}$ is the extremal graph we ultimately need. To help readers better understand our constructions, each of the following paragraphs will focus on the construction of a specific graph.

We denote by $G(i,j)$ the complete graph on five vertices whose edges are colored with two colors, $i$ and $j$, such that the edges of color $i$ and the edges of color $j$ each induce a monochromatic cycle $C_5$.

Next, we construct $H_{2k}$. We denote $G(1,2)$ as $H_2$. For $k \ge 2$, the graph $H_{2k}$ is obtained by replacing each vertex of the graph $G(2k-1,2k)$ with a copy of the graph $H_{2k-2}$. It is known that $H_{2k}$ is, in fact, the extremal graph of $GR_{2k}(K_3)$. Clearly, we have $|H_{2k}|=5^k$.

Next, we construct $F(p,q)$. For $m \ge 8$ and $m \equiv 0 \pmod{4}$, we denote $F(p,q)$ as the graph with $m-2$ vertices, which consists of two copies of the complete graph with $m/2-1$ vertices, where all edges within each copy are colored $p$, and all edges between the two copies are colored $q$. For $m \ge 6$ and $m \equiv 2 \pmod{4}$, we define $F(p,q)$ as a complete graph with $m-1$ vertices that is $2$-edge colored, such that the induced subgraph of the edges of color $p$ and the induced subgraph of the edges of color $q$ both form $(m/2-1)$-regular graphs. Since $m/2-1$ is even, such a coloring is guaranteed to exist. For instance, when $m=6$, we have $F(p,q) = G(p,q)$. It is straightforward to verify that for any even $m \ge 6$, the induced subgraph formed by all edges of color $p$ in $F(p,q)$ has a maximum degree of at most $m/2-1$, and the number of vertices in each connected component is at most $m-1$. The same properties hold for the induced subgraph formed by all edges of color $q$ in $F(p,q)$.

Having constructed the graphs $H_{2k}$ and $F(p,q)$, we now proceed to create their composite graph $H_{2k}[F(p,q)]$. We define $H_0[F(p,q)] = F(p,q)$. For $k \ge 1$, the graph $H_{2k}[F(p,q)]$ is obtained by replacing each vertex of the graph $H_{2k}$ with a copy of the graph $F(p,q)$. Therefore, we have $|H_{2k}[F(p,q)]| = (m-2)5^k$ for $m \ge 8$ and $m \equiv 0 \pmod{4}$; and $|H_{2k}[F(p,q)]| = (m-1)5^k$ for $m \ge 6$ and $m \equiv 2 \pmod{4}$.

\begin{figure}[H]
  \centering
  \scalebox{0.7}{
\begin{tikzpicture}

  \def\radius{2cm}

  \def\pentagonRadius{3cm}

  \def\lineWidth{15pt}

  \coordinate (A) at ({\pentagonRadius*cos(90)}, {\pentagonRadius*sin(90)});
  \coordinate (B) at ({\pentagonRadius*cos(162)}, {\pentagonRadius*sin(162)});
  \coordinate (C) at ({\pentagonRadius*cos(234)}, {\pentagonRadius*sin(234)});
  \coordinate (D) at ({\pentagonRadius*cos(306)}, {\pentagonRadius*sin(306)});
  \coordinate (E) at ({\pentagonRadius*cos(18)}, {\pentagonRadius*sin(18)});

  \draw[red!50, line width=\lineWidth] (A) -- (B);
  \draw[red!50, line width=\lineWidth] (B) -- (C);
  \draw[red!50, line width=\lineWidth] (C) -- (D);
  \draw[red!50, line width=\lineWidth] (D) -- (E);
  \draw[red!50, line width=\lineWidth] (E) -- (A);

  \draw[blue!50, line width=\lineWidth] (A) -- (C);
  \draw[blue!50, line width=\lineWidth] (C) -- (E);
  \draw[blue!50, line width=\lineWidth] (E) -- (B);
  \draw[blue!50, line width=\lineWidth] (B) -- (D);
  \draw[blue!50, line width=\lineWidth] (D) -- (A);

  \foreach \i in {A,B,C,D,E}{
      \node[draw, fill=white, circle, minimum size=\radius] at (\i) {};
  }

  \def\vertexSize{8pt}  

  \begin{scope}
    \coordinate (V1) at ($(A)+({0.6*cos(90)},{0.6*sin(90)})$);
    \coordinate (V2) at ($(A)+({0.6*cos(210)},{0.6*sin(210)})$);
    \coordinate (V3) at ($(A)+({0.6*cos(330)},{0.6*sin(330)})$);
    \draw[line width=2pt, black] (V1) -- (V2) -- (V3) -- cycle;
    \node at (V1) [draw, circle, fill=white, minimum size=\vertexSize, inner sep=0pt] {};
    \node at (V2) [draw, circle, fill=white, minimum size=\vertexSize, inner sep=0pt] {};
    \node at (V3) [draw, circle, fill=white, minimum size=\vertexSize, inner sep=0pt] {};
  \end{scope}

  \begin{scope}
    \coordinate (V4) at ($(B)+({0.7*cos(0)},{0.7*sin(0)})$);
    \coordinate (V5) at ($(B)+({0.7*cos(72)},{0.7*sin(72)})$);
    \coordinate (V6) at ($(B)+({0.7*cos(144)},{0.7*sin(144)})$);
    \coordinate (V7) at ($(B)+({0.7*cos(216)},{0.7*sin(216)})$);
    \coordinate (V8) at ($(B)+({0.7*cos(288)},{0.7*sin(288)})$);
    \draw[red, line width=2pt] (V4) -- (V5) -- (V6) -- (V7) -- (V8) -- cycle;
    \draw[line width=2pt] (V4) -- (V6) -- (V8) -- (V5) -- (V7) -- cycle;
    \node at (V4) [draw, circle, fill=white, minimum size=\vertexSize, inner sep=0pt] {};
    \node at (V5) [draw, circle, fill=white, minimum size=\vertexSize, inner sep=0pt] {};
    \node at (V6) [draw, circle, fill=white, minimum size=\vertexSize, inner sep=0pt] {};
    \node at (V7) [draw, circle, fill=white, minimum size=\vertexSize, inner sep=0pt] {};
    \node at (V8) [draw, circle, fill=white, minimum size=\vertexSize, inner sep=0pt] {};
  \end{scope}

  \begin{scope}
    \coordinate (V9) at ($(C)+({0.7*cos(0)},{0.7*sin(0)})$);
    \coordinate (V10) at ($(C)+({0.7*cos(72)},{0.7*sin(72)})$);
    \coordinate (V11) at ($(C)+({0.7*cos(144)},{0.7*sin(144)})$);
    \coordinate (V12) at ($(C)+({0.7*cos(216)},{0.7*sin(216)})$);
    \coordinate (V13) at ($(C)+({0.7*cos(288)},{0.7*sin(288)})$);
    \draw[blue, line width=2pt] (V9) -- (V10) -- (V11) -- (V12) -- (V13) -- cycle;
    \draw[line width=2pt] (V9) -- (V11) -- (V13) -- (V10) -- (V12) -- cycle;
    \node at (V9) [draw, circle, fill=white, minimum size=\vertexSize, inner sep=0pt] {};
    \node at (V10) [draw, circle, fill=white, minimum size=\vertexSize, inner sep=0pt] {};
    \node at (V11) [draw, circle, fill=white, minimum size=\vertexSize, inner sep=0pt] {};
    \node at (V12) [draw, circle, fill=white, minimum size=\vertexSize, inner sep=0pt] {};
    \node at (V13) [draw, circle, fill=white, minimum size=\vertexSize, inner sep=0pt] {};
  \end{scope}

  \begin{scope}
    \coordinate (V14) at ($(D)+({0.7*cos(0)},{0.7*sin(0)})$);
    \coordinate (V15) at ($(D)+({0.7*cos(72)},{0.7*sin(72)})$);
    \coordinate (V16) at ($(D)+({0.7*cos(144)},{0.7*sin(144)})$);
    \coordinate (V17) at ($(D)+({0.7*cos(216)},{0.7*sin(216)})$);
    \coordinate (V18) at ($(D)+({0.7*cos(288)},{0.7*sin(288)})$);
    \draw[blue, line width=2pt] (V14) -- (V15) -- (V16) -- (V17) -- (V18) -- cycle;
    \draw[line width=2pt] (V14) -- (V16) -- (V18) -- (V15) -- (V17) -- cycle;
    \node at (V14) [draw, circle, fill=white, minimum size=\vertexSize, inner sep=0pt] {};
    \node at (V15) [draw, circle, fill=white, minimum size=\vertexSize, inner sep=0pt] {};
    \node at (V16) [draw, circle, fill=white, minimum size=\vertexSize, inner sep=0pt] {};
    \node at (V17) [draw, circle, fill=white, minimum size=\vertexSize, inner sep=0pt] {};
    \node at (V18) [draw, circle, fill=white, minimum size=\vertexSize, inner sep=0pt] {};
  \end{scope}

  \begin{scope}
    \coordinate (V19) at ($(E)+({0.7*cos(0)},{0.7*sin(0)})$);
    \coordinate (V20) at ($(E)+({0.7*cos(72)},{0.7*sin(72)})$);
    \coordinate (V21) at ($(E)+({0.7*cos(144)},{0.7*sin(144)})$);
    \coordinate (V22) at ($(E)+({0.7*cos(216)},{0.7*sin(216)})$);
    \coordinate (V23) at ($(E)+({0.7*cos(288)},{0.7*sin(288)})$);
    \draw[red, line width=2pt] (V19) -- (V20) -- (V21) -- (V22) -- (V23) -- cycle;
    \draw[line width=2pt] (V19) -- (V21) -- (V23) -- (V20) -- (V22) -- cycle;
    \node at (V19) [draw, circle, fill=white, minimum size=\vertexSize, inner sep=0pt] {};
    \node at (V20) [draw, circle, fill=white, minimum size=\vertexSize, inner sep=0pt] {};
    \node at (V21) [draw, circle, fill=white, minimum size=\vertexSize, inner sep=0pt] {};
    \node at (V22) [draw, circle, fill=white, minimum size=\vertexSize, inner sep=0pt] {};
    \node at (V23) [draw, circle, fill=white, minimum size=\vertexSize, inner sep=0pt] {};
  \end{scope}
\end{tikzpicture}}
\caption{The graph $G_2(\ell)$ for $m=6$, where red, blue, and black represent colors $1$,  $2$, and $\ell$, respectively.}
\label{fig:G2ell}
\end{figure}
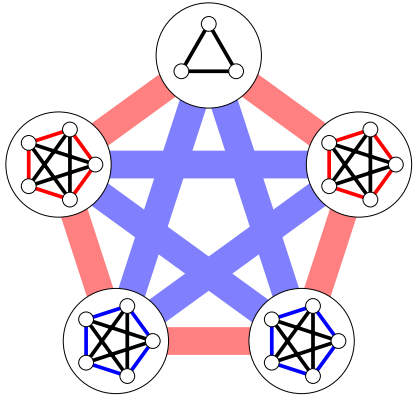

For a positive integer $k$, we construct $G_{2k-2}(\ell)$ as follows, where $\ell$ does not belong to the color set $[2k-2]$. When $k=1$, the graph $G_0(\ell)$ is a $K_{m/2}$, where all edges are colored $\ell$. When $k \ge 2$, the graph $G_{2k-2}(\ell)$ is obtained by blowing up $G(2k-3, 2k-2)$. We blow up a vertex $v$ of $G(2k-3, 2k-2)$ into $G_{2k-4}(\ell)$, the two vertices adjacent to $v$ through edges of color $2k-3$ are expanded into $H_{2k-4}[F(\ell,2k-3)]$, and the two vertices adjacent to $v$ through edges of color $2k-2$ are expanded into $H_{2k-4}[F(\ell,2k-2)]$. Therefore, for $k \ge 2$, we have
\[
|G_{2k-2}(\ell)| = |G_{2k-4}(\ell)| + 4 \cdot (m-2)5^{k-2} \quad \text{for}\ m \ge 8\ \text{and}\ m \equiv 0\ (\bmod{\ 4}),
\]
and
\[
|G_{2k-2}(\ell)| = |G_{2k-4}(\ell)| + 4 \cdot (m-1)5^{k-2} \quad \text{for}\ m \ge 6\ \text{and}\ m \equiv 2\ (\bmod{\ 4}).
\]
Furthermore, it can be derived that
\[
|G_{2k-2}(\ell)| = (m-2)5^{k-1} - m/2 + 2 \quad \text{for}\ m \ge 8\ \text{and}\ m \equiv 0\ (\bmod{4}),
\]
and
\[
|G_{2k-2}(\ell)| = (m-1)5^{k-1} - m/2 + 1 \quad \text{for}\ m \ge 6\ \text{and}\ m \equiv 2\ (\bmod{4}).
\]

Finally, we construct the extremal graph $G_{2k}$ that we require. When $k=1$, the extremal graph $G_2$ is a complete graph on $R(\widehat{K}_m, \widehat{K}_m)-1$ vertices with a 2-edge-coloring, where each edge is colored either 1 or 2, and there is no monochromatic $\widehat{K}_m$. In other words, $G_2$ is the extremal graph corresponding to the 2-color Ramsey number $R(\widehat{K}_m, \widehat{K}_m)$. Hence, $|G_2| = R(\widehat{K}_m, \widehat{K}_m) - 1$.

When $k \ge 2$, the extremal graph $G_{2k}$ is obtained by blowing up $G(2k-1, 2k)$. We blow up a vertex $v$ into $G_{2k-2}$, the two vertices adjacent to $v$ via edges of color $2k-1$ are expanded into $G_{2k-2}(2k-1)$, and the two vertices adjacent to $v$ via edges of color $2k$ are expanded into $G_{2k-2}(2k)$. Thus, for $m \ge 8$ and $m \equiv 0 \pmod{4}$, we have
\begin{equation}\label{equation_2}
|G_{2k}| = |G_{2k-2}| + 4(m-2)5^{k-1} - 2m + 8 = R_2(\widehat{K}_m) - 1 + 5(m-2)(5^{k-1} - 1) - 2(m-4)(k-1)\,;
\end{equation}
and for $m \ge 6$ and $m \equiv 2 \pmod{4}$, we have
\begin{equation}\label{equation_3}
|G_{2k}| = |G_{2k-2}| + 4(m-1)5^{k-1} - 2m + 4 = R_2(\widehat{K}_m) - 1 + 5(m-1)(5^{k-1} - 1) - 2(m-2)(k-1)\,.
\end{equation}

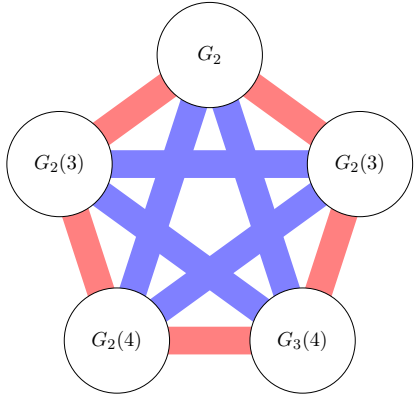
\begin{figure}[H]
  \centering
  \scalebox{0.7}{
\begin{tikzpicture}

  \def\radius{2cm}

  \def\pentagonRadius{3cm}

  \def\lineWidth{15pt}

  \coordinate (A) at ({\pentagonRadius*cos(90)}, {\pentagonRadius*sin(90)});
  \coordinate (B) at ({\pentagonRadius*cos(162)}, {\pentagonRadius*sin(162)});
  \coordinate (C) at ({\pentagonRadius*cos(234)}, {\pentagonRadius*sin(234)});
  \coordinate (D) at ({\pentagonRadius*cos(306)}, {\pentagonRadius*sin(306)});
  \coordinate (E) at ({\pentagonRadius*cos(18)}, {\pentagonRadius*sin(18)});

  \draw[red!50, line width=\lineWidth] (A) -- (B);
  \draw[red!50, line width=\lineWidth] (B) -- (C);
  \draw[red!50, line width=\lineWidth] (C) -- (D);
  \draw[red!50, line width=\lineWidth] (D) -- (E);
  \draw[red!50, line width=\lineWidth] (E) -- (A);

  \draw[blue!50, line width=\lineWidth] (A) -- (C);
  \draw[blue!50, line width=\lineWidth] (C) -- (E);
  \draw[blue!50, line width=\lineWidth] (E) -- (B);
  \draw[blue!50, line width=\lineWidth] (B) -- (D);
  \draw[blue!50, line width=\lineWidth] (D) -- (A);

  \foreach \i in {A,B,C,D,E}{
      \node[draw, fill=white, circle, minimum size=\radius] at (\i) {};
  }

  \node at (A) {$G_2$};
  \node at (B) {$G_2(3)$};
  \node at (C) {$G_2(4)$};
  \node at (D) {$G_3(4)$};
  \node at (E) {$G_2(3)$};
\end{tikzpicture}}
\caption{The graph $G_4$, where red and blue represent colors $3$ and $4$, respectively.}
\label{fig:G4}
\end{figure}

It is clear that there is no monochromatic $\widehat{K}_m$ in $G_2$. For $k \ge 2$, the graph $G_{2k}$ is constructed via a blow-up procedure. Thus, it is easy to see that there are no rainbow triangles. We verify that $G_{2k}$ contains no monochromatic $\widehat{K}_m$ through the following four claims.

\begin{claim}\label{basicproperty}
  For $k \ge 2$ and even $m \ge 6$, in the subgraph of $G_{2k-2}(\ell)$ induced by edges of color $\ell$, the degree of each vertex is at most $m/2-1$, and the number of vertices in each connected component is no more than $m-1$.
\end{claim}

\begin{proof}
  By definition, the graph $G_0(\ell)$ is a complete graph on $m/2$ vertices, with all edges colored $\ell$. In both $F(\ell, 2k-3)$ and $F(\ell, 2k-2)$, each vertex is adjacent to at most $m/2-1$ edges of color $2\ell$, and the number of vertices in both graphs is no more than $m-1$. Based on the blow-up construction of $G_{2k-2}(\ell)$, this claim follows easily.
\end{proof}

\begin{claim}\label{2k-12k}
  For $2k-1 \le i \le 2k$, there is no $\widehat{K}_m$ of color $i$ in $G_{2k}$.
\end{claim}

\begin{proof}
  We only need to prove the case for $i = 2k$. By symmetry, the other case can be similarly proven. The color set used in $G_{2k-2}$ is $\{1, 2, \ldots, 2k-2\}$, so there are no edges of color $2k$ in it. The color set used in $G_{2k-2}(2k-1)$ is $\{1, 2, \ldots, 2k-2, 2k-1\}$, so it also contains no edges of color $2k$. By Claim~\ref{basicproperty}, in the subgraph of $G_{2k-2}(2k)$ induced by edges of color $2k$, the degree of each vertex is at most $m/2-1$, and the number of vertices in each connected component is no more than $m-1$. Now, representing the color $2k$ with red, by Lemma~\ref{basiclemma}, there is no $\widehat{K}_m$ of color $2k$ in $G_{2k}$.
\end{proof}

\begin{claim}\label{2k-22k}
  For $k \ge 2$ and each $i \in [2k-2]$, there is no $\widehat{K}_m$ of color $i$ in $G_{2k-2}(2k-1)$ and $G_{2k-2}(2k)$.
\end{claim}

\begin{proof}
  We only need to prove that there is no $\widehat{K}_m$ of color $i$ in $G_{2k-2}(\ell)$ for each $i \in [2k-2]$, where the color $\ell$ is not in $[2k-2]$.
  
  First, consider the graph $H_{2k-4}[F(\ell,2k-3)]$. If it contains a triangle of color $i$ for some $i \in [2k-4]$, then $H_{2k-4}$ would also contain a triangle of color $i$, which leads to a contradiction. In the subgraph of $H_{2k-4}[F(\ell,2k-3)]$ induced by edges of color $2k-3$, the degree of each vertex is at most $m/2-1$, and each connected component contains no more than $m-1$ vertices. Therefore, for each $i \in [2k-2]$, there is no $\widehat{K}_m$ of color $i$ in $H_{2k-4}[F(\ell,2k-3)]$. By a similar argument, there is no $\widehat{K}_m$ of color $i$ in $H_{2k-4}[F(\ell,2k-2)]$.
  
  Now, let us proceed with mathematical induction on the positive integer $k$. First, we prove the base case. When $k=2$, the graph $G_2(\ell)$ is obtained by blowing up $G(1,2)$, where its five vertices are expanded into one $G_0(\ell)$, two $F(\ell,1)$, and two $F(\ell,2)$. All edges between the two $F(\ell,1)$ sets are of color $2$, and all edges between the two $F(\ell,2)$ sets are of color $1$. The number of vertices in $F(\ell,1)$ is no more than $m-1$, and each vertex is incident to at most $m/2-1$ edges of color $1$. Assigning red to represent the $1$-color, by Lemma~\ref{basiclemma}, there is no $\widehat{K}_m$ of color $1$ in $G_2(\ell)$. Similarly, $G_2(\ell)$ contains no $\widehat{K}_m$ of color $2$.
  
  When $k \ge 3$, assume that there is no $\widehat{K}_m$ of color $i$ in $G_{2(k-1)-2}(\ell)$ for each $i \in [2(k-1)-2]$. Now consider the graph $G_{2k-2}(\ell)$. We know that $G_{2k-2}(\ell)$ is obtained by blowing up $G(2k-3, 2k-2)$. From the above analysis, a $\widehat{K}_m$ of color $i$ cannot appear in any of the five parts. Therefore, if a monochromatic $\widehat{K}_m$ exists, it must be either color $2k-3$ or color $2k-2$.
  
  We map the color $2k-3$ to correspond to the red color in Lemma~\ref{basiclemma}. By applying the lemma, it follows that there is no $\widehat{K}_m$ with color $2k-3$. By symmetry, the same argument applies to show that no monochromatic $\widehat{K}_m$ exists with color $2k-2$. Thus, there is no $\widehat{K}_m$ of color $i$ in $G_{2k-2}(\ell)$ for each $i \in [2k-2]$.
\end{proof}

\begin{claim}\label{finalclaim}
  For each $i \in [2k]$, there is no $\widehat{K}_m$ of color $i$ in $G_{2k}$.
\end{claim}

\begin{proof}
  We proceed by induction on the positive integer $k$. When $k=1$, it is clear that there is no monochromatic $\widehat{K}_m$ in $G_2$. For $k \ge 2$, assume that for each $i \in [2(k-1)]$, there is no $\widehat{K}_m$ of color $i$ in $G_{2(k-1)}$. Now, consider the graph $G_{2k}$. By Claim~\ref{2k-12k}, there is no monochromatic $\widehat{K}_m$ of color $2k-1$ or $2k$ in $G_{2k}$.
  
  Next, we use proof by contradiction. Suppose there exists an $i \in [2k-2]$ such that there is a $\widehat{K}_m$ of color $i$ in $G_{2k}$. Since $G_{2k}$ is obtained by inflating $G(2k-1, 2k)$, the $\widehat{K}_m$ of color $i$ can only be contained in $G_{2k-2}$, $G_{2k-2}(2k-1)$, or $G_{2k-2}(2k)$. By the inductive hypothesis, the $\widehat{K}_m$ of color $i$ cannot be contained in $G_{2k-2}$. By Claim~\ref{2k-22k}, there is no $\widehat{K}_m$ of color $i$ in either $G_{2k-2}(2k-1)$ or $G_{2k-2}(2k)$. This contradiction completes the proof.
\end{proof}

According to Claim~\ref{finalclaim}, for even $k \ge 2$ and even $m \ge 6$, we have $GR_k(\widehat{K}_m)\ge |G_k|+1$. By combining this with Equations~(\ref{equation_2}) and (\ref{equation_3}), we obtain our lemma.    
\end{proof}

Next, we consider the case where $k$ is odd. We need to prove the following result.
\begin{lemma}
  For odd $k \ge 3$ and even $m \ge 6$, we have
    \begin{equation*}
      GR_k(\widehat{K}_m) \ge  
      \begin{cases}
      \max\{2R_2(G)-1, 5m+1\}, & \text{if } k=3;\\
      2(R_2(\widehat{K}_m)+5(m-2)(5^{\frac{k-2}{2}}-1)-(m-4)(k-2))-1, & \text{if } k\ge 5 \text{ is odd and } m \equiv 0\ (\bmod{\ 4});\\
      2(R_2(\widehat{K}_m)+5(m-1)(5^{\frac{k-2}{2}}-1)-(m-2)(k-2))-1, & \text{if } k\ge 5 \text{ is odd and } m \equiv 2\ (\bmod{\ 4}).
      \end{cases}
    \end{equation*}
\end{lemma}
\begin{proof}
We first verify that $GR_3(\widehat{K}_m) \ge 5m + 1$. We replace each vertex of $G(1,2)$ in the previous lemma with the complete graph $K_m$ of color $3$. It is evident that this graph contains neither a monochromatic $\widehat{K}_m$ nor a rainbow triangle. Thus, we have $GR_3(\widehat{K}_m)\ge 5m+1$.

Since $k-1$ is odd, we can derive the graph $G_{k-1}$ based on Equations~(\ref{equation_2}) and (\ref{equation_3}). We construct two copies of this graph, with all edges between them colored with color $k$. Clearly, this graph contains neither a monochromatic $\widehat{K}_m$ nor a rainbow triangle. Therefore, we obtain the lower bound $GR_k(\widehat{K}_m) \ge 2|G_{k-1}| + 1$. This leads to all the other lower bounds.
\end{proof}

At the end of this section, we consider the case where $m$ is odd and $m \ge 7$, which corresponds to proving Theorem~\ref{thmoddm}.
\begin{proof}[Proof of Theorem~\ref{thmoddm}]
It is evident that $m-1$ is an even number of at least $6$. When $k$ is even, we can construct the extremal graph $G_k$ corresponding to $m-1$ as stated in Theorem~\ref{thmkipas}. A slight adjustment is required: we can replace the unique $G_2$ in $G_k$, which corresponds to the extremal graph of $R_2(\widehat{K}_{m-1})$, with the extremal graph of $R_2(\widehat{K}_m)$. 

When $k$ is odd, we still follow the construction from Theorem~\ref{thmkipas} to obtain the corresponding extremal graph for $m-1$. This graph includes two copies of the extremal graph for $R_2(\widehat{K}_{m-1})$, and we replace both with the extremal graph for $R_2(\widehat{K}_m)$. This completes the proof of the theorem.
\end{proof}

\section{Gallai-Ramsey non-full graphs}\label{secfull}
In this section, we discuss the conjecture of Su and Liu~\cite{Su2022}, which states that for any graph $G$ without isolated vertices, $G$ is Ramsey-full if and only if $G$ is Gallai-Ramsey-full. One easily verifiable observation is that if the latter proposition ($G$ is Gallai-Ramsey-full) holds, then it can be deduced from $GR_2(G)=R(G,G)$ that the former proposition ($G$ is Ramsey-full) must also hold. However, the condition of being Ramsey-full does not necessarily imply being Gallai-Ramsey-full. For instance, Erd\H{o}s and Faudree~\cite{Erdos1992} noted that $K_{1,n}$ is Ramsey-full for even $n$. Nonetheless, Su and Liu~\cite{Su2022} have provided a rigorous proof that $GR^*_k(K_{1,n})=2n-2$ for even $n\ge 12$. Combining this result with the fact that $GR_k(K_{1,n})=5n/2-3$ for $k\ge 3$ and even $n\ge 4$, it can be concluded that the conjecture made by Su and Liu is incorrect.

\begin{lemma}
  For every even integer $n$ with $n\ge 12$, $K_{1,n}$ is Ramsey-full but not Gallai-Ramsey-full.
\end{lemma}

The star with an even number of edges is not an isolated counterexample. We can also provide the following counterexample. We use the notation $K^+_{1,n}$ to represent the graph obtained by subdividing one edge of $K_{1,n}$. Hence, $K^+_{1,n}$ is a tree with $n+2$ vertices and maximum degree $n$. Based on the following two lemmas, it can be concluded that $K^+_{1,n}$ is Ramsey-full but not Gallai-Ramsey-full.

\begin{lemma}\label{K1n+}
  For every even integer $n$ with $n\ge 4$, $K^+_{1,n}$ is Ramsey-full.
\end{lemma}
\begin{proof}
 Assume $n$ is an even number with $n \ge 4$. From $K_{2n-1}-e \not\to (K_{1,n},K_{1,n})$, we can deduce that $K_{2n-1}-e \not\to (K^+_{1,n},K^+_{1,n})$. Therefore, we only need to prove that $K_{2n-1} \to (K^+_{1,n},K^+_{1,n})$. Since $K_{2n-1} \to (K_{1,n},K_{1,n})$, for any 2-edge-coloring of $K_{2n-1}$, there must exist a monochromatic $K_{1,n}$, which we may assume is red. Denote the set of $n$ non-center vertices of $K_{1,n}$ as $X$, and the set of vertices not in $K_{1,n}$ as $Y$. If there exists a red edge between $X$ and $Y$, then this edge, combined with the red $K_{1,n}$, forms a red $K^+_{1,n}$. If all edges between $X$ and $Y$ are blue, then there exists a blue $K_{n,n-2}$, which contains $K^+_{1,n}$ as a subgraph. This completes the proof.
\end{proof}

\begin{lemma}\label{main5}
For any $k\ge 3$ and every even integer $n$ with $n\ge 12$, \[GR_k(K^+_{1,n})=5n/2+k-6 \text{ and } GR^*_k(K^+_{1,n})=2n+k-5\,.\] Therefore, $K^+_{1,n}$ is not Gallai-Ramsey-full.
\end{lemma}

\begin{proof}
In the following proof, we use red, blue, and green to refer to the colors $1$, $2$, and $3$, respectively.

We denote a colored $K_5$ by $G(1,2)$, where the edges are colored with red and blue. Moreover, the red edges and the blue edges each induce a monochromatic pentagon. Subsequently, we blow up one of the vertices to form a complete graph with $n/2$ vertices, while the remaining four vertices are blown up to form complete graphs with $n/2-1$ vertices each. Note that we have not specified the colors of edges in these five complete subgraphs.

Next, we introduce a new vertex $v_1$ and connect $v_1$ to every existing vertex with an edge of color $4$. We then proceed to add another vertex $v_2$ and connect $v_2$ to every existing vertex with an edge of color $5$. This process continues in the same manner until we have added $k-3$ vertices, namely $v_1, \ldots, v_{k-3}$. We denote the resulting graph as $H$. As shown in Figure~\ref{fig:Ramseyfull}, we denote the five complete subgraphs obtained after the blown-up operation as $V_1, \ldots, V_5$, where $|V_1|=n/2$ and $|V_i|=n/2-1$ for $i\in [2,5]$.

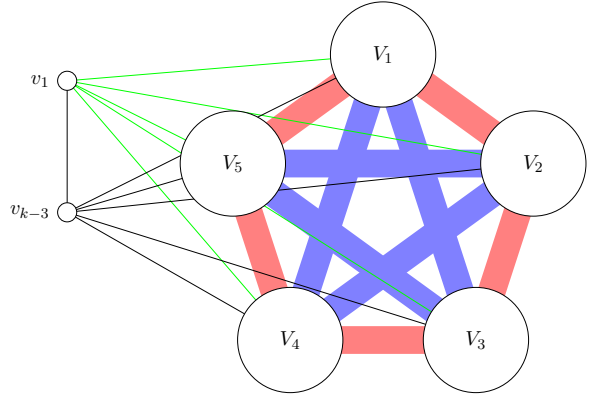
\begin{figure}[H]
  \centering
  \scalebox{0.7}{
  \begin{tikzpicture}

    \def\radius{2cm}

    \def\pentagonRadius{3cm}

    \def\lineWidth{15pt}

    \coordinate (A) at ({\pentagonRadius*cos(90)}, {\pentagonRadius*sin(90)});
    \coordinate (B) at ({\pentagonRadius*cos(162)}, {\pentagonRadius*sin(162)});
    \coordinate (C) at ({\pentagonRadius*cos(234)}, {\pentagonRadius*sin(234)});
    \coordinate (D) at ({\pentagonRadius*cos(306)}, {\pentagonRadius*sin(306)});
    \coordinate (E) at ({\pentagonRadius*cos(18)}, {\pentagonRadius*sin(18)});

    \coordinate (v1) at (-6, 2.5);
    \coordinate (vk) at (-6, 0);

    \draw[red!50, line width=\lineWidth] (A) -- (B);
    \draw[red!50, line width=\lineWidth] (B) -- (C);
    \draw[red!50, line width=\lineWidth] (C) -- (D);
    \draw[red!50, line width=\lineWidth] (D) -- (E);
    \draw[red!50, line width=\lineWidth] (E) -- (A);

    \draw[blue!50, line width=\lineWidth] (A) -- (C);
    \draw[blue!50, line width=\lineWidth] (C) -- (E);
    \draw[blue!50, line width=\lineWidth] (E) -- (B);
    \draw[blue!50, line width=\lineWidth] (B) -- (D);
    \draw[blue!50, line width=\lineWidth] (D) -- (A);

    \draw[green] (v1) -- (A);
    \draw[green] (v1) -- (B);
    \draw[green] (v1) -- (C);
    \draw[green] (v1) -- (D);
    \draw[green] (v1) -- (E);
    
    \draw[black] (vk) -- (v1);
    \draw[black] (vk) -- (A);
    \draw[black] (vk) -- (B);
    \draw[black] (vk) -- (C);
    \draw[black] (vk) -- (D);
    \draw[black] (vk) -- (E);

    \foreach \i in {A,B,C,D,E}{
        \node[draw, fill=white, circle, minimum size=\radius] at (\i) {};
    }

    \node[draw, fill=white, circle, minimum size=3pt] at (v1) {};
    \node[draw, fill=white, circle, minimum size=3pt] at (vk) {};

    \node at (A) {$V_1$};
    \node at (B) {$V_5$};
    \node at (C) {$V_4$};
    \node at (D) {$V_3$};
    \node at (E) {$V_2$};
    
    \node[left=2mm] at (v1) {$v_1$};
    \node[left=2mm] at (vk) {$v_{k-3}$};
  \end{tikzpicture}}
  \caption{The graph $H$, where the edges in $V_i$ are uncolored for each $i\in [5]$.}
  \label{fig:Ramseyfull}
\end{figure}

\begin{claim}
  $GR_k(K^+_{1,n}) \ge \frac{5n}{2}+k-6$ and $GR^*_k(K^+_{1,n}) \ge 2n+k-5$.
\end{claim}
\begin{proof}
We color the edges within each $G[V_i]$ with color $3$ for each $i \in [5]$. In this way, the coloring of the complete graph $H$ is completed. It can be observed that $H$ is a Gallai coloring, and $H$ does not contain a monochromatic $K_{1,n}$ in color $i$ for each $i \in [3]$. The edges of color $j$ induce a star for each $j \in [4,k]$. Therefore, $H$ does not contain a monochromatic $K^+_{1,n}$. Hence, $GR_k(K^+_{1,n}) \ge \frac{5n}{2}+k-6$.

Next, we add a vertex $w$. Connect $w$ to each $v_i$ for $i \in [k-3]$, and color the edge $wv_i$ with color $i+3$. Connect $w$ to every vertex in $V_1 \cup V_3$, and color these edges red. Connect $w$ to every vertex in $V_2 \cup V_5$, and color these edges blue. At this point, $w$ is connected by $2n+k-6$ edges. In the new graph, it still does not contain a monochromatic $K_{1,n}$ in color $i$ for each $i \in [3]$, and the edges of color $j$ induce a star for each $j \in [4,k]$. Therefore, the new graph still does not contain a monochromatic $K^+_{1,n}$. Thus, $GR^*_k(K^+_{1,n}) \ge 2n+k-5$.
\end{proof}

To prove the upper bound, we need to establish the following claim. Note that the edges within each $G[V_i]$ in the graph $H$ are uncolored.

\begin{claim}
For any Gallai $k$-edge coloring of a complete graph with $5n/2+k-7$ vertices, if it does not contain a monochromatic $K^+_{1,n}$, then, after appropriately rearranging the colors, its coloring is the same as that of the graph $H$.
\end{claim}

\begin{proof}
Let $(G, \tau)$ be a Gallai $k$-edge coloring of a complete graph on $5n/2+k-7$ vertices. Assume that $G$ does not contain any monochromatic $K^+_{1,n}$. The objective is to demonstrate that, by appropriately arranging the colors, the coloring of $G$ will coincide with that of the graph $H$.

Consider a Gallai partition of the graph $G$ such that the number of parts is as small as possible. If there are only two parts and the smaller part contains only one vertex, denoted by $v_1$, we then consider the graph $G-v_1$. Next, consider a Gallai partition of $G-v_1$ such that the number of parts is minimized. Again, if there are only two parts and the smaller part contains only one vertex, denoted by $v_2$, we proceed by considering the graph $G-\{v_1, v_2\}$. This process continues until we have performed this operation $k-2$ times, or until it can no longer be done. Suppose this process is carried out $t$ times in total, where $t \le k-2$.

If $t=k-2$, then $|G-\{v_1, \ldots, v_{k-2}\}|=5n/2-5 \ge n + 1$ for $n \ge 4$. We now examine the colors of the edges between $v_1, \ldots, v_{k-2}$ and $G - \{v_1, \ldots, v_{k-2}\}$. According to the process above, for each $i \in [k-2]$, all edges from $v_i$ to $G - \{v_1, \ldots, v_{k-2}\}$ are of the same color. If there exist $i, j$ with $1 \le i < j \le k-2$ such that the color of edges from $v_i$ to $G - \{v_1, \ldots, v_{k-2}\}$ is the same as that from $v_j$ to $G - \{v_1, \ldots, v_{k-2}\}$, then by selecting $n$ vertices from $G - \{v_1, \ldots, v_{k-2}\}$ and combining them with $v_i$ and $v_j$, we would obtain a monochromatic $K^+_{1,n}$.

If, for any $i, j$ with $1 \le i < j \le k-2$, the colors of edges from $v_i$ to $G-\{v_1, \ldots, v_{k-2}\}$ are different, then we have used $k-2$ distinct colors. Without loss of generality, assume these $k-2$ colors are $1, 2, \ldots, k-2$. Consequently, none of the edges in $G - \{v_1, \ldots, v_{k-2}\}$ can use these $k-2$ colors. Otherwise, suppose there exists an edge $u_1u_2 \in E(G - \{v_1, \ldots, v_{k-2}\})$ colored with color 1. In this case, $v_1u_1u_2$ forms a monochromatic $P_3$ of color 1. By selecting $n-1$ vertices from $G - (\{v_1, \ldots, v_{k-2}\} \cup \{u_1, u_2\})$ and combining them with $v_1$, we obtain a monochromatic $K_{1,n-1}$ of color 1, which, together with $u_1$ and $u_2$, forms a $K^+_{1,n}$, leading to a contradiction.

The graph $G - \{v_1, \ldots, v_{k-2}\}$ is a complete graph on $5n/2-5$ vertices, but all edges can only be colored with two colors. According to $R(K^+_{1,n}, K^+_{1,n}) = 2n - 1$ for $n \ge 4$, and since $5n/2-5 \ge 2n - 1$ for $n \ge 8$, we can find a monochromatic $K^+_{1,n}$ in $G - \{v_1, \ldots, v_{k-2}\}$, leading to a contradiction.

If $t \le k - 3$, we denote the graph $G - \{v_1, \ldots, v_t\}$ as $G'$. Consider a Gallai partition of the graph $G'$ that minimizes the number of parts. The edges between the parts can have at most two colors, which we may assume to be red and blue. If there are only two parts, then according to the construction of $G'$, the smaller part must contain at least two vertices, while the larger part contains at least $5n/4 - 2$ vertices. For $n \ge 5$, we have $\left\lceil 5n/4 \right\rceil - 2 \ge n$. We can select $n$ vertices from the larger part and 2 vertices from the smaller part to form a monochromatic $K_{2,n}$, which contains $K^+_{1,n}$ as a subgraph. 

If the Gallai partition consists of exactly three parts, note that the edges between any two parts have only one color, which allows the graph $G'$ to be divided into two parts. This contradicts the minimality of the number of parts. Therefore, we assume that the Gallai partition has at least four parts. Suppose one of the parts contains only a single vertex, denoted as $v$. The edges connecting vertex $v$ to the other vertices in $G'$ can only be of two colors: red and blue. Without loss of generality, assume the number of red edges connected to $v$ is at least as large as the number of blue edges. Consequently, all the red edges associated with $v$ induce a red star with at least $\left\lceil 5n/4 - 5/2 \right\rceil$ edges, which is at least $n$ for $n \ge 7$. 

Let $S$ be the set of vertices in $G'$ excluding those that are the vertices of this red star. To avoid the existence of a red $K^+_{1,n}$, any vertex $u$ in $S$ must connect with all vertices in $V(G') \setminus (S \cup \{v\})$ using blue edges. If $|S| = 1$, let $S = \{u\}$. Clearly, the edge $uv$ is blue. Thus, $G'$ has a Gallai partition with only two parts: $S$ and $V(G') \setminus S$, leading to a contradiction. If $|S| \ge 2$, we can find a blue $K_{2,n}$, which also contains $K^+_{1,n}$ as a subgraph, resulting in a contradiction. Therefore, a Gallai partition of $G'$ must consist of at least four parts, each containing at least two vertices.

Since $t \le k - 3$, we have $|G'| \ge 5n/2 - 4$. We claim that $|V_i| \ge n/2 - 2$ for each part $V_i$. This is because if there exists a part, say $V_1$, with at most $n/2 - 3$ vertices, then the remaining parts must contain at least $2n - 1$ vertices. By the pigeonhole principle, each vertex in $V_1$ must be connected to either at least $n$ red edges or at least $n$ blue edges. Since this is a Gallai coloring, the edge connection pattern of each vertex in $V_1$ to the outside must be the same. Moreover, because $|V_1| \ge 2$, we can easily derive a monochromatic $K^+_{1,n}$, leading to a contradiction. This proves that $|V_i| \ge n/2 - 2$.

If $G'$ has exactly four parts, let us denote them as $V_1, V_2, V_3, V_4$, with $|V_1| \ge |V_2| \ge |V_3| \ge |V_4|$. Without loss of generality, assume that the edges in $E(V_4, V_1)$ are colored red. Then, the edges in $E(V_4, V_3)$ must be colored blue. If not, since $|V_1| + |V_3| \ge |G'|/2 \ge n$, there would be a red $K_{2,n}$ in $E(V_4, V_1 \cup V_3)$, which contains $K^+_{1,n}$ as a subgraph, leading to a contradiction. 

Similarly, we can prove that the edges in $E(V_4, V_2)$ must also be colored blue. If $|V_2| + |V_3| \ge n$, then there would be a blue $K_{2,n}$ in $E(V_4, V_2 \cup V_3)$, resulting in another contradiction. Therefore, we set $|V_2| + |V_3| \le n - 1$. This implies that $|V_1| + |V_4| \ge n$. 

To avoid having the edges between $V_2$ and $V_1 \cup V_4$ containing a blue $K_{2,n}$ as a subgraph, the edges between $V_1$ and $V_2$ must be red. Likewise, all edges between $V_1$ and $V_3$ must also be red. Consequently, all edges between $V_1$ and $V_2 \cup V_3 \cup V_4$ are red, allowing $G'$ to be divided into two parts. This contradicts the assumption that $G'$ has at least four parts. Therefore, $G'$ must have at least five parts.

If $G'$ has at least six parts, by the pigeonhole principle, there must be at least three parts whose edges with $V_1$ are of the same color. Without loss of generality, assume that the edges between $V_1$ and $V_2 \cup V_3 \cup V_4$ are all red. Since $|V_i| \ge n/2 - 2$, it follows that $|V_2 \cup V_3 \cup V_4| \ge 3n/2 - 6 \ge n$. Therefore, the edges between $V_1$ and $V_2 \cup V_3 \cup V_4$ contain a red $K_{2,n}$ as a subgraph, which leads to a contradiction. This indicates that $G'$ must have exactly five parts.

To avoid the occurrence of a monochromatic $K_{2,n}$, the reduced graph of $G'$ must consist of two monochromatic cycles of five vertices. Assume that the largest part has $n/2 + k$ vertices, which implies $k \ge 0$; otherwise, there would be a contradiction with the total number of vertices in $G'$, which is at least $5n/2 - 4$. To prevent the emergence of a monochromatic $K_{2,n}$, the sum of the number of vertices in any two parts cannot exceed $n - 1$. Thus, each part other than $V_1$ must have at most $n/2 - k - 1$ vertices. Consequently, the total number of vertices is at most $5n/2 - 3k - 4$. From the inequality $5n/2 - 4 \le 5n/2 - 3k - 4$, it follows that $k = 0$. Therefore, the number of vertices in $G'$ is exactly $5n/2 - 4$. This leads to $t = k - 3$, and one of the five parts has $n/2$ vertices, while the other four parts each have $n/2 - 1$ vertices. It is now straightforward to verify that, with an appropriate arrangement of colors, the coloring of graph $G$ is the same as that of graph $H$.

We denote the five parts of $G'$ as $V_1, V_2, V_3, V_4, V_5$. Let us set $|V_1| = n/2$ and $|V_i| = n/2 - 1$ for $i \in [2, 5]$. For $i \in [5]$, we can assume that the edges between $V_i$ and $V_{i+1}$ are red edges (that is, edges of color $1$), while the edges between $V_i$ and $V_{i+2}$ are blue edges (that is, edges of color $2$). Here, the indices are taken modulo five. We denote the vertices in $V(G) \setminus V(G')$ as $v_1, v_2, \ldots, v_{k-3}$, and we assume that all edges between $v_i$ and $G'$ are of color $i + 3$. In this way, the coloring of graph $G$ is completely consistent with that of graph $H$.
\end{proof}

Now consider the last vertex $w$. Since $w$ is connected to $2n + k - 5$ edges, there are at least $2n - 2$ edges between $w$ and $G'$. For the edges between $w$ and $G'$, we have the following two claims.

\begin{claim}
  There are no red edges in $E(w,V_2\cup V_5)$ and no blue edges in $E(w,V_3\cup V_4)$.
\end{claim}
\begin{proof}
  We only need to prove the former, as the latter can be proven similarly. Let us assume that $wv_2$ is a red edge, where $v_2\in V_2$. Since there are $n-1$ red edges between $v_2$ and $V_1\cup V_3$, there exists a red $K_{1,n}$ centered at $v_2$. Now, let us choose a vertex $v_2'$ from $V_2\setminus \{v_2\}$. Notice that all edges between $v_2'$ and $V_1\cup V_3$ are also red. Therefore, it is easy to find a red $K^+_{1,n}$, leading to a contradiction.
\end{proof}

\begin{claim}
  All edges between the vertex $w$ and $G'$ are either red or blue.
\end{claim}
\begin{proof}
First, suppose there is an edge of color $i$ between $w$ and $G'$ for some $i \ge 4$. Since all edges between $v_{i-3}$ and $G'$ are colored $i$, and $|G'| \ge n$, the graph induced by $w$, $v_{i-3}$, and $V(G')$ contains a monochromatic $K^+_{1,n}$ in color $i$, leading to a contradiction.

Next, suppose there is an edge of color 3 between $w$ and $G'$. By symmetry, we consider two subcases: the color 3 edge appears in $E(w, V_1)$ or in $E(w, V_4)$. According to the previous claim, there are no red edges in $E(w, V_2 \cup V_5)$ and no blue edges in $E(w, V_3 \cup V_4)$. 

For the first subcase, in order to avoid a rainbow triangle, all edges in $E(w, V_2 \cup V_5)$ must be colored 3, and all edges in $E(w, V_3 \cup V_4)$ must also be colored 3. Thus, if any $G[V_i]$ contains an edge of color 3 for some $i \in [2,5]$, we will obtain a monochromatic $K^+_{1,n}$ in color 3. Therefore, all edges in $G[V_i]$ can only be red or blue for each $i \in [2,5]$. Since $n \ge 12$, we have $|V_2| =n/2-1 \ge 5$. By $R(P_4, P_4) = 5$, there must be either a red $P_4$ or a blue $P_4$ in $G[V_2]$. In either case, we can find a monochromatic $K^+_{1,n}$ in $G'$. 

Now consider the second subcase where the color 3 edge appears in $E(w, V_4)$. To avoid a rainbow triangle, all edges in $E(w, V_5)$ must be colored 3; next, all edges in $E(w, V_3)$ must also be colored 3; similarly, all edges in $E(w, V_2)$ and $E(w, V_4)$ must be colored 3, and finally, all edges in $E(w, V_1)$ must be colored 3 as well. This reduces the second subcase to the first one, completing the proof. 

Thus, all edges between $w$ and $G'$ must be either red or blue.
\end{proof}

Based on the two claims above, the edges in $E(w, V_2\cup V_5)$ can only be blue, and the edges in $E(w, V_3\cup V_4)$ can only be red. According to the results by Su and Liu~\cite{Su2022}, there must exist a monochromatic $K_{1,n}$. In the sequel, we will look for a monochromatic $K^+_{1,n}$. We divide our discussion into two cases based on the location of the center of $K_{1,n}$. The first case is when there exists a monochromatic $K_{1,n}$ with its center in $G'$. Without loss of generality, let us assume this is a red $K_{1,n}$ with the center vertex denoted as $u$ and the set of leaf vertices denoted as $U$. Based on the structure of the graph $H$, it is easy to select a vertex, denoted as $u'$, in $V(G')\setminus (\{u\}\cup U)$ such that there is at least one red edge between $u'$ and $U$. Thus, the graph induced by $U, u,$ and $u'$ contains a red subgraph $K^+_{1,n}$. The second case is when there exists a monochromatic $K_{1,n}$ with $w$ as its center. Without loss of generality, let us assume this star is red, and let $W$ denote the set of its $n$ leaf vertices. Since $|V_1|=n/2$, $|V_3|=|V_4|=n/2-1$, and $|W|=n$, it follows that $W\cap V_i\neq\emptyset$ for each $i\in \{1,3,4\}$. Based on the structure of the graph $H$, it is feasible to select a vertex, denoted as $w'$, from $V(G')\setminus W$, ensuring the existence of at least one red edge connecting $w'$ with $W$. Thus, the graph induced by $W, w,$ and $w'$ contains a red subgraph $K^+_{1,n}$. For readers who require further elucidation, the proof of Theorem 4 in Su and Liu~\cite{Su2022} can provide additional insights.
\end{proof}

Combining the three lemmas above, the proof of Theorem~\ref{Ramseyfull} is complete.

\subsection*{Acknowledgements}

The first author was partially supported by NSFC under grant number 11601527. The second author was partially supported by NSFC under grant numbers 2161141003 and 11931006.

\end{document}